\documentclass[11pt]{scrartcl}

\usepackage[utf8]{inputenc}
\usepackage[T1]{fontenc}
\usepackage{lmodern}
\usepackage{alphabeta}

\usepackage{titling}
\usepackage{xspace}
\usepackage{soul} 

\usepackage[letterpaper, top=25.4mm, bottom=25.4mm, left=25.4mm, right=25.4mm, includefoot]{geometry}

\DeclareSectionCommand[%
level=4,
indent=0pt,
beforeskip=1ex plus 1ex minus .2ex,
afterskip=-1em,
font={},
tocindent=7em,
tocnumwidth=4.1em,
counterwithin=subsubsection
]{paragraph}

\usepackage[inline]{enumitem}

\usepackage{dsfont}
\usepackage{setspace}

\usepackage{bm}
\usepackage{microtype}
\usepackage{amsmath}
\usepackage{amssymb}
\usepackage{amsfonts}
\usepackage{mathtools}
\usepackage[mathscr]{euscript}

\usepackage[hyphens]{url}
\usepackage{amsthm}

\usepackage{tikz}
\usepackage{xcolor}
\usepackage{subcaption}
\usepackage{graphicx}
\usepackage{wrapfig}

\usepackage{hyperref}
\usepackage{zref-clever}

\usepackage{nicefrac}
\usepackage[textwidth=1.5in]{todonotes}

\colorlet{myGreen}{green!50!black}
\colorlet{myLightgreen}{green}
\colorlet{myRed}{red!90!black}
\definecolor{myBlue}{rgb}{0.25, 0.0, 1.0}
\definecolor{myLightBlue}{rgb}{0.39, 0.58, 0.93}
\colorlet{myViolet}{myBlue!55!myRed}
\definecolor{myOrange}{rgb}{1.0, 0.66, 0.07}

\definecolor{CornflowerBlue}{rgb}{0.39, 0.58, 0.93}
\definecolor{DarkGoldenrod}{rgb}{0.72, 0.53, 0.04}
\definecolor{BritishRacingGreen}{rgb}{0.0, 0.26, 0.15}
\definecolor{DarkMagenta}{rgb}{0.55, 0.0, 0.55}
\definecolor{AO}{rgb}{0.0, 0.5, 0.0}
\definecolor{BostonUniversityRed}{rgb}{0.8, 0.0, 0.0}
\definecolor{myRed}{rgb}{0.8, 0.0, 0.0}
\definecolor{DarkMidnightBlue}{rgb}{0.0, 0.2, 0.4}
\definecolor{DarkTangerine}{rgb}{1.0, 0.66, 0.07}
\definecolor{AppleGreen}{rgb}{0.55, 0.71, 0.0}
\definecolor{BrightUbe}{rgb}{0.82, 0.62, 0.91}
\definecolor{Amethyst}{rgb}{0.6, 0.4, 0.8}
\definecolor{DarkGray}{rgb}{0.52, 0.52, 0.51}
\definecolor{Gray}{rgb}{0.66, 0.66, 0.66}
\definecolor{BananaYellow}{rgb}{1.0, 0.88, 0.21}
\definecolor{Amber}{rgb}{1.0, 0.75, 0.0}
\definecolor{LightGray}{rgb}{0.83, 0.83, 0.83}
\definecolor{PrincetonOrange}{rgb}{1.0, 0.56, 0.0}
\definecolor{DeepCarrotOrange}{rgb}{0.91, 0.41, 0.17}
\definecolor{CarrotOrange}{rgb}{0.93, 0.57, 0.13}
\definecolor{MidnightBlue}{rgb}{0.1, 0.1, 0.44}
\definecolor{Magenta}{rgb}{0.50, 0.0, 0.50}
\definecolor{BrightPink}{rgb}{1.0, 0.0, 0.5}
\definecolor{BrilliantRose}{rgb}{1.0, 0.33, 0.64}
\definecolor{ChromeYellow}{rgb}{1.0, 0.65, 0.0}
\definecolor{HotMagenta}{rgb}{1.0, 0.11, 0.81}
\definecolor{Auburn}{rgb}{0.43, 0.21, 0.1}
\definecolor{BrightTurquoise}{rgb}{0.03, 0.91, 0.87}
\definecolor{DarkCyan}{rgb}{0.0, 0.55, 0.55}

\vfuzz \hfuzz
\setlist[itemize]{topsep=0pt,partopsep=0pt,itemsep=0pt,parsep=0pt}
\setlist[itemize,1]{label={\small\textbullet}}
\setlist[itemize,2]{label={\tiny\textbullet}}
\setlist[itemize,3]{label=$\cdot$}
\setlist[enumerate]{topsep=0pt,partopsep=0pt,itemsep=0pt,parsep=0pt}
\setlist[enumerate,1]{label=\roman*)}
\setlist[enumerate,2]{label=\alph*)}
\setlist[enumerate,3]{label=\arabic*)}

\hypersetup{
colorlinks=true,
linkcolor=AO!65!black,
citecolor=AO!65!black,
urlcolor=AppleGreen!65!black,
bookmarksopen=true,
bookmarksnumbered,
bookmarksopenlevel=2,
bookmarksdepth=3
}

\newcommand*\samethanks[1][\value{footnote}]{\footnotemark[#1]}

\theoremstyle{definition}

\newtheorem{environment}{Environment}[section]

\newtheorem{lemma}[environment]{Lemma}
\AddToHook{env/lemma/begin}{\zcsetup{countertype={environment=lemma}}}
\zcRefTypeSetup{lemma}{
Name-sg = Lemma ,
name-sg = Lemma ,
Name-pl = Lemmas ,
name-pl = Lemmas ,
}

\newtheorem*{lemma*}{Lemma}
\AddToHook{env/lemma*/begin}{\zcsetup{countertype={environment=lemma*}}}
\zcRefTypeSetup{lemma*}{
Name-sg = Lemma ,
name-sg = Lemma ,
Name-pl = Lemmas ,
name-pl = Lemmas ,
}

\AddToHook{env/proposition/begin}{\zcsetup{countertype={environment=proposition}}}
\zcRefTypeSetup{proposition}{
Name-sg = Proposition ,
name-sg = Proposition ,
Name-pl = Propositions ,
name-pl = Propositions ,
}

\newtheorem{corollary}[environment]{Corollary}
\AddToHook{env/corollary/begin}{\zcsetup{countertype={environment=corollary}}}
\zcRefTypeSetup{corollary}{
Name-sg = Corollary ,
name-sg = Corollary ,
Name-pl = Corollaries ,
name-pl = Corollaries ,
}

\newtheorem{theorem}[environment]{Theorem}
\AddToHook{env/theorem/begin}{\zcsetup{countertype={environment=theorem}}}
\zcRefTypeSetup{theorem}{
Name-sg = Theorem ,
name-sg = Theorem ,
Name-pl = Theorems ,
name-pl = Theorems ,
}

\newtheorem*{theorem*}{Theorem}
\AddToHook{env/theorem*/begin}{\zcsetup{countertype={environment=theorem*}}}
\zcRefTypeSetup{theorem*}{
Name-sg = Theorem ,
name-sg = Theorem ,
Name-pl = Theorems ,
name-pl = Theorems ,
}

\AddToHook{env/conjecture/begin}{\zcsetup{countertype={environment=conjecture}}}
\zcRefTypeSetup{conjecture}{
Name-sg = Conjecture ,
name-sg = Conjecture ,
Name-pl = Conjectures ,
name-pl = Conjectures ,
}

\newtheorem*{hypothesis*}{Hypothesis}
\AddToHook{env/hypothesis*/begin}{\zcsetup{countertype={environment=hypothesis*}}}
\zcRefTypeSetup{hypothesis*}{
Name-sg = Hypothesis ,
name-sg = Hypothesis ,
Name-pl = Hypotheses ,
name-pl = Hypotheses ,
}

\AddToHook{env/observation/begin}{\zcsetup{countertype={environment=observation}}}
\zcRefTypeSetup{observation}{
Name-sg = Observation ,
name-sg = Observation ,
Name-pl = Observations ,
name-pl = Observations ,
}

\AddToHook{env/example/begin}{\zcsetup{countertype={environment=example}}}
\zcRefTypeSetup{example}{
Name-sg = Example ,
name-sg = Example ,
Name-pl = Examples ,
name-pl = Examples ,
}

\AddToHook{env/remark/begin}{\zcsetup{countertype={environment=remark}}}
\zcRefTypeSetup{remark}{
Name-sg = Remark ,
name-sg = Remark ,
Name-pl = Remarks ,
name-pl = Remarks ,
}

\zcRefTypeSetup{equation}{
Name-sg = Equation ,
name-sg = Equation ,
Name-pl = Equations ,
name-pl = Equations ,
}

\zcRefTypeSetup{chapter}{
Name-sg = Chapter ,
name-sg = Chapter ,
Name-pl = Chapters ,
name-pl = Chapters ,
}

\zcRefTypeSetup{section}{
Name-sg = Section ,
name-sg = Section ,
Name-pl = Sections ,
name-pl = Sections ,
}

\zcRefTypeSetup{algorithm}{
Name-sg = Algorithm ,
name-sg = Algorithm ,
Name-pl = Algorithms ,
name-pl = Algorithms ,
}

\AddToHook{env/notation/begin}{\zcsetup{countertype={environment=notation}}}
\zcRefTypeSetup{notation}{
Name-sg = Notation ,
name-sg = Notation ,
Name-pl = Notations ,
name-pl = Notations ,
}

\AddToHook{env/question/begin}{\zcsetup{countertype={environment=question}}}
\zcRefTypeSetup{question}{
Name-sg = Question ,
name-sg = Question ,
Name-pl = Questions ,
name-pl = Questions ,
}

\newtheorem{problem}[environment]{Problem}
\AddToHook{env/problem/begin}{\zcsetup{countertype={environment=problem}}}
\zcRefTypeSetup{problem}{
Name-sg = Problem ,
name-sg = Problem ,
Name-pl = Problems ,
name-pl = Problems ,
}

\AddToHook{env/claim/begin}{\zcsetup{countertype={environment=claim}}}
\zcRefTypeSetup{claim}{
Name-sg = Claim ,
name-sg = Claim ,
Name-pl = Claims ,
name-pl = Claims ,
}

\AddToHook{env/definition/begin}{\zcsetup{countertype={environment=definition}}}
\zcRefTypeSetup{definition}{
Name-sg = Definition ,
name-sg = Definition ,
Name-pl = Definitions ,
name-pl = Definitions ,
}

\zcRefTypeSetup{figure}{
Name-sg = Figure ,
name-sg = Figure ,
Name-pl = Figures ,
name-pl = Figures ,
}

\usetikzlibrary{calc}
\usetikzlibrary{fit}
\usetikzlibrary{decorations}
\usetikzlibrary{decorations.pathmorphing}
\usetikzlibrary{decorations.text}
\usetikzlibrary{shapes,hobby}

\tikzset{
	position/.style args={#1:#2 from #3}{
		at=($(#3)+(#1:#2)$)
	}
}

\tikzset{
  v:main/.style = {draw, circle, scale=0.8, thick,fill=black,inner sep=0.7mm},
  v:ghost/.style = {inner sep=0pt,scale=1},
  >={latex},
  e:marker/.style = {line width=8.5pt,line cap=round,opacity=0.35,color=DarkGoldenrod},
  e:main/.style = {line width=1pt},
}

\title{The Erd\H{o}s-P\'osa property for prime-length cycles fails (and beyond)}
\predate{}
\date{}
\postdate{}

\preauthor{}
\DeclareRobustCommand{\authorthing}{
	\begin{center}
		Maximilian Gorsky\thanks{Supported by the Institute for Basic Science (IBS-R029-C1).}~~\!\footnote{\href{mailto:m.gorsky@pm.me}{m.gorsky@pm.me}} \\
		{\small Discrete Mathematics Group, Institute for Basic Science (IBS), Daejeon, South Korea} \\
        \medskip
        Kevin Hendrey\thanks{Supported by the Australian Research Council.}~~\!\footnote{\href{mailto:Kevin.Hendrey1@monash.edu}{Kevin.Hendrey1@monash.edu}} \\
		{\small School of Mathematics, Monash University, Melbourne, Australia} \\
        \medskip
        Tony Huynh\samethanks[1]~~\!\footnote{\href{mailto:tony@ibs.re.kr}{tony@ibs.re.kr}} \\
		{\small Discrete Mathematics Group, Institute for Basic Science (IBS), Daejeon, South Korea}
\end{center}}
\author{\authorthing}
\postauthor{}

\begin{document}
\maketitle

\begin{abstract}
We prove that for every $t \in \mathbb{N}$, prime-length cycles do not have the $\frac{1}{t}$-integral Erd\H{o}s-P\'osa property, even when restricted to planar graphs.  We in fact prove a more general density result.  For every $t \in \mathbb{N}$ and every subset $L \subseteq \mathbb{N}$ with lower density zero, the set of cycles whose length is in $L$ do not have the $\frac{1}{t}$-integral Erd\H{o}s-P\'osa property, even when restricted to planar graphs.  We also consider a less restrictive density condition on $L$, called \emph{porous}, where the complement of $L$ contains arbitrarily long sequences of consecutive integers.  We prove that for every porous set $L \subseteq \mathbb{N}$, the set of cycles whose length is in $L$ do not have the Erd\H{o}s-P\'osa property, even when restricted to projective planar graphs.  
Our results partially answer a question of Gollin, Hendrey, Kwon, Oum, and Yoo~\cite{GHKOY25}.  
\end{abstract}
\let\sc\itshape
\thispagestyle{empty}

\newpage

\newpage
\thispagestyle{empty}

\newpage

\setcounter{page}{1}

\section{Introduction}
 In 1965, Erd\H{o}s and P\'osa~\cite{ErdosP1965Independent} proved that there exists a function $f(k)$ such that for all $k \in \mathbb{N}$ and all graphs $G$, either $G$ contains $k$ vertex-disjoint cycles or there exists a set $X \subseteq V(G)$ with $|X| \leq f(k)$ such that $G-X$ is a forest.  This theorem has been hugely influential in structural graph theory (and beyond) and has spawned an entire research area broadly known as the \emph{Erd\H{o}s-P\'osa property} (to be defined below).  These results are far too numerous to list here, but we refer the interested reader to~\cite{RaymondT2017Recent} for a general survey on the Erd\H{o}s-P\'osa property. 

In particular, there is an extensive line of research on the Erd\H{o}s-P\'osa property for cycles satisfying various length constraints including triangles~\cite{tuza90}, odd cycles~\cite{Reed1999Mangoes,KawarabayashiKKX2025Halfintegral}, even cycles~\cite{BruhnHJ2018Frames,GorskyKKW2024Packing,Gorsky2024Structure}, cycles of length $0 \pmod{m}$~\cite{Thomassen1988Presence},  cycles of length $\ell \pmod{p}$ with $p$ prime~\cite{ThomasY2023Packinga}, and long cycles~\cite{BBR07, CJU20, MNSW17}.  All of the above results can be put into a unified framework which we now describe. 

Let $L \subseteq \mathbb{N}$.  An \emph{$L$-cycle} is a cycle whose length is in $L$.  We say that \emph{$L$-cycles have the Erd\H{o}s-P\'osa property} if there exists a function $f(k)$ such that for all $k \in \mathbb{N}$ and all graphs $G$, either $G$ contains $k$ vertex-disjoint $L$-cycles, or there exists $X \subseteq V(G)$ such that $|X| \leq f(k)$ and $G-X$ has no $L$-cycle.  
More generally, for $t \in \mathbb{N}$, we say that \emph{$L$-cycles have the $\frac{1}{t}$-integral Erd\H{o}s-P\'osa property} if there exists a function $f(k)$ such that for all $k \in \mathbb{N}$ and all graphs $G$,  either $G$ contains a set of $k$ $L$-cycles such that each $v \in V(G)$ is in at most $t$ of these cycles, or there exists $X \subseteq V(G)$ such that $|X| \leq f(k)$ and $G-X$ has no $L$-cycle.  For example, Reed~\cite{Reed1999Mangoes} proved that odd cycles do not have the Erd\H{o}s-P\'osa property, but that they do have the $\frac{1}{2}$-integral Erd\H{o}s-P\'osa property.  Gollin, Hendrey, Kwon, Oum, and Yoo~\cite{GHKOY25} posed the following very general problem. 

\begin{problem} \label{characterize}
Characterize the sets $L \subseteq \mathbb{N}$ such that the set of $L$-cycles have the Erd\H{o}s-P\'osa property.  
\end{problem}

It is easy to see that $L$-cycles have the Erd\H{o}s-P\'osa property whenever $L$ is finite.  Thus,~\zcref{characterize} is only interesting when $L$ is infinite. In particular, the case when $L$ is the set of primes has received considerable attention.\footnote{The prime-length cycle question has been posed at several talks and workshops, most recently by O-joung Kwon at the \emph{Focused Workshop on Erdős-Pósa problems} in Będlewo, Poland, March 15-20, 2026.} Here, we provide the following strong negative answer. 

\begin{theorem} \label{primefailure}
For every $t \in \mathbb{N}$, prime-length cycles do not have the $\frac{1}{t}$-integral Erd\H{o}s-P\'osa property, even when restricted to planar graphs.
\end{theorem}

We actually prove a stronger result, which only requires a density condition on $L$.   We define the \emph{lower density of $L$} to be $\underline{d}(L):=\liminf_{n \to \infty} \frac{|L \cap \{1, \dots, n\}|}{n}$.   

\begin{theorem} \label{lowerdensity}
For every infinite subset $L \subseteq \mathbb{N}$ with $\underline{d}(L)=0$, and every $t \in \mathbb{N}$, the set of $L$-cycles do not have the $\frac{1}{t}$-integral Erd\H{o}s-P\'osa property, even when restricted to planar graphs.
\end{theorem}

Note that by the Prime Number Theorem~\cite{deLaValleePoussin1896, Hadamard1896}, the set of primes have lower density zero. Thus, \zcref{primefailure} is an immediate corollary of~\zcref{lowerdensity}.

We also consider a relaxed density condition on $L$.  We say that $L$ is \emph{porous} if $\mathbb{N} \setminus L$ contains $n$ consecutive integers for every $n \in \mathbb{N}$.  It is easy to see that every lower density zero set is porous, but that there are porous sets which have positive lower density.  For porous sets $L$, we rule out even the following weak form of the Erd\H{o}s-P\'osa property.

\begin{theorem} \label{porous}
For every infinite porous set $L$, there does not exist a function $f(k)$ such that for all $k \in \mathbb{N}$ and all graphs $G$, either $G$ contains $k$ vertex-disjoint $L$-cycles, or a set of at most $f(k)$ vertices $X$ such that $G-X$ contains at most $f(k)$ cycle lengths from $L$.
\end{theorem}

\section{Lower density zero sets}
In this section, we will prove~\zcref{lowerdensity}.  For $s,t \in \mathbb{Z}$, we let $[s,t]:=\{s, s+1, \dots, t\}$ and $[t]:=[1, \dots, t]$. We begin with a simple lemma on constructing special subsets of integers which avoid lower density zero sets.

\begin{lemma}\label{lem:algebra}
    Let $L \subseteq \mathbb{N}$ have lower density zero.  Then for every $t \in \mathbb{N}$ there is some $x_t\in \mathbb{N}$ such that $L\cap \{ax_t+b:a\in [t], b\in [-t, t]\}=\emptyset$.
\end{lemma}
\begin{proof}
Suppose this fails for some $t$.
For each $x\in \mathbb{N}$ choose $a \in [t]$ and $b \in [-t,t]$ such that $ax+b\in L$.  Let the colour of $x$ be $(a,b)$.  Since $\underline{d}(L)=0$, there exists $N \geq (2t+1)^3$ such that $|L\cap [N]|/N< \frac{1}{(2t+1)^3}$.
Since there are only $t(2t+1)$ colours, some colour appears at least $\frac{N}{t(2t+1)^2}$ times among the first $\frac{N}{2t+1}$ elements of $[N]$.  Thus, $|L \cap [N]| \geq \frac{N}{t(2t+1)^2}$. However, this contradicts $|L\cap [N]|/N< \frac{1}{(2t+1)^3}$.
\end{proof}

For each positive integer $\ell$, the \emph{$(\ell \times \ell)$-grid} is a graph with the vertex set $[\ell] \times [\ell]$, where for all $i,j,i',j' \in [\ell]$ the vertex $(i,j)$ is adjacent to $(i',j')$ if and only if $|i-i'|+|j-j'|=1$.  We are now ready to prove~\zcref{lowerdensity} in the following form.  

\begin{theorem} \label{sparsegrid}
    Let $L \subseteq \mathbb{N}$ be an infinite set with lower density zero. 
    For all $t,s \in \mathbb{N}$, there exists a planar graph $G$ such that
    \begin{enumerate}
        \item For every collection of $t$ $L$-cycles in $G$, there is a vertex in their common intersection, and,
        \item For all $S\subseteq V(G)$ with $|S|\leq s$, $G-S$ contains an $L$-cycle.
    \end{enumerate}
\end{theorem}
\begin{proof}
For each $x\in\mathbb{N}$, let $g(x)$ be the minimum positive integer such that $g(x)x+b\in L$, for some $b \in [-g(x), g(x)]$.   Note that every positive integer can be written in the form $qx+r$ with $0 \leq r \leq x$.  Thus, $g(x)$ is finite, since $L$ contains elements larger than $x^2$.  
By Lemma~\ref{lem:algebra}, there exists $x$ with $g(x)$ arbitrarily large.  
Thus, we may choose $x$ such that $\frac{g(x)^{1/2}}{4t} > s$, and there exists an $\ell \in \mathbb{N}$ such that $\frac{2tg(x)}{2t-1} \leq \ell^2 \leq \frac{3tg(x)}{3t-2}$.

Let $G_{\ell \times \ell}$ be the $(\ell \times \ell)$-grid and let $G$ be obtained from $G_{\ell \times \ell}$ by replacing each edge with three internally disjoint paths of lengths $x-1$, $x$, and $x+1$, respectively.  
By the choice of $g(x)$, every $L$-cycle in $G$ corresponds to a cycle of length at least $g(x)$ in $G_{\ell \times \ell}$. Since $\frac{tg(x)}{|V(G_{\ell \times \ell})|} > t-1$, every collection of $t$ $L$-cycles in $G$ has a vertex in their common intersection.

For the second part, towards a contradiction, let $S \subseteq V(G)$ be such that $|S|=s$ and $G-S$ has no $L$-cycle.   We may clearly assume $S \subseteq V(G_{\ell \times \ell})$. Observe that $G_{\ell \times \ell}-S$ contains a subdivision of $G_{(\ell-s) \times (\ell-s)}$.  Thus, $G_{\ell \times \ell}-S$ contains a cycle $C$ of length at least 
\[
(\ell-s)^2 \geq \left( \left(\frac{2t}{2t-1}\right)^{1/2} g(x)^{1/2} - \frac{1}{4t} g(x)^{1/2} \right)^2 \geq g(x),  
\]
where the last inequality follows from the easy observation that $\left(\frac{2t}{2t-1}\right)^{1/2} - \frac{1}{4t} \geq 1$, for all $t \in \mathbb{N}$.  
Because $C$ lifts to an $L$-cycle in $G-S$, we conclude that $G-S$ contains an $L$-cycle, which is a contradiction.   Thus, $|S| > s$, as required.   
\end{proof}

\section{Porous sets}

In this section we will prove~\zcref{porous}.  We begin with a simple lemma on constructing special subsets of integers which avoid porous sets.



\begin{lemma} \label{lem:farfromL}
Let $L$ be a porous subset of $\mathbb{N}$. Then there is an infinite set $S\subseteq \mathbb{N}$ such that for every nonempty finite $S'\subseteq S$, the sum of the elements of $S'$ is not in $L$.
\end{lemma}


\begin{proof}
For $X \subseteq \mathbb{N}$, let $\sum X$ denote the sum of the elements in $X$.  
We say that $X$ is \emph{far from $L$} if $\sum X' \notin L$ for all nonempty finite $X' \subseteq X$.  Let $Z$ be an arbitrary finite subset of $\mathbb{N}$ which is far from $L$.  It suffices to show that there exists $y \notin Z$ such that $Z \cup \{y\}$ is far from $L$.  Let $\ell:=\sum_{z \in Z} z$.  Since there are arbitrarily long intervals in $\mathbb{N}\setminus L$, there exists $y > \ell$ such that $[y, y+\ell] \cap L = \emptyset$.  Thus, $Z \cup \{y\}$ is far from $L$, as required.
\end{proof}

We are now ready to prove~\zcref{porous} in the following form.

\begin{theorem} \label{thm:porousgrid}
    Let $L$ be an infinite porous subset of $\mathbb{N}$.  
    For each $\ell \geq 1$ there exists a projective planar graph $G$ such that $G$ does not contain two vertex-disjoint $L$-cycles and $G-X$ contains at least $\ell$ cycle lengths in $L$ for all $X \subseteq V(G)$ with $|X| \leq \ell$. 
\end{theorem}

\begin{proof} 
    By subdividing edges, it suffices to consider edge-weighted graphs.  We construct $G$ and $w:E(G) \to \mathbb{N}$ as follows.  Start with the  $(6\ell \times 6\ell)$-grid $W$.
    By Lemma~\ref{lem:farfromL}, there exists an infinite subset $S \subseteq \mathbb{N}$ which is far from $L$. Choose an arbitrary subset $A$ of $S$ of size $|E(W)|$ and let $w:E(W) \to A$ be an arbitrary bijection. 
    
    For each $i \in [3\ell]$ let $a_i:=(1,i)$, $b_i:=(6\ell, 6\ell+1-i)$, and $Q_i$ be an $a_i$-$b_i$ path in $W$ such that every vertex of $W$ is in at most two such paths.   
    Say that an element $p \in L$ is \emph{$k$-lonely} if $q-p \geq k$, where $q$ is the smallest element of $L$ larger than $p$.  Since $L$ is porous, for each $k \in \mathbb{N}$, there are arbitrarily large $k$-lonely elements.  
    We construct a sequence of elements $p_1, \dots, p_{3\ell} \in L$ recursively as follows.  Let $\alpha$ be the sum of the elements in $A$.  Choose $p_1 \in L$ such that $p_1 \geq 2\alpha$  and $p_1$ is $\alpha$-lonely. Recursively choose $p_k$ so that $p_k > p_{k-1}$ and $p_k$ is $(\sum_{i=1}^{k-1} p_i + \alpha)$-lonely.  

    For each $i \in [3\ell]$ add an edge $e_i$ between $a_i:=(1,i)$ and $b_i:=(6\ell, 6\ell+1-i)$, and let $w(e_i):=p_i-w(Q_i)$, where $w(Q_i)=\sum_{e \in E(Q_i)} w(e)$.  See Figure~\ref{fig:escherwall} for an illustration of the graph $G$ and the paths $Q_1, \dots, Q_{3k}$.  
    Let $X \subseteq V(G)$ with $|X| \leq \ell$.  Since $|X| \leq \ell$ and each $x \in X$ intersects at most two of the paths $Q_1, \dots, Q_{3\ell}$, there exists $I \subseteq [3\ell]$ such that $|I|=\ell$ and $X \cap V(Q_i)=\emptyset$ for all $i \in I$. Thus, for each $i \in I$, the cycle $Q_i \cup e_i$ in $G-X$ has weight $p_i \in L$.  

    \begin{figure}[ht]
    \centering
        \begin{tikzpicture}[scale=1.25]

            \pgfdeclarelayer{background}
		      \pgfdeclarelayer{foreground}
			
		      \pgfsetlayers{background,main,foreground}

            \begin{pgfonlayer}{background}
            \pgftext{\includegraphics[width=6cm]{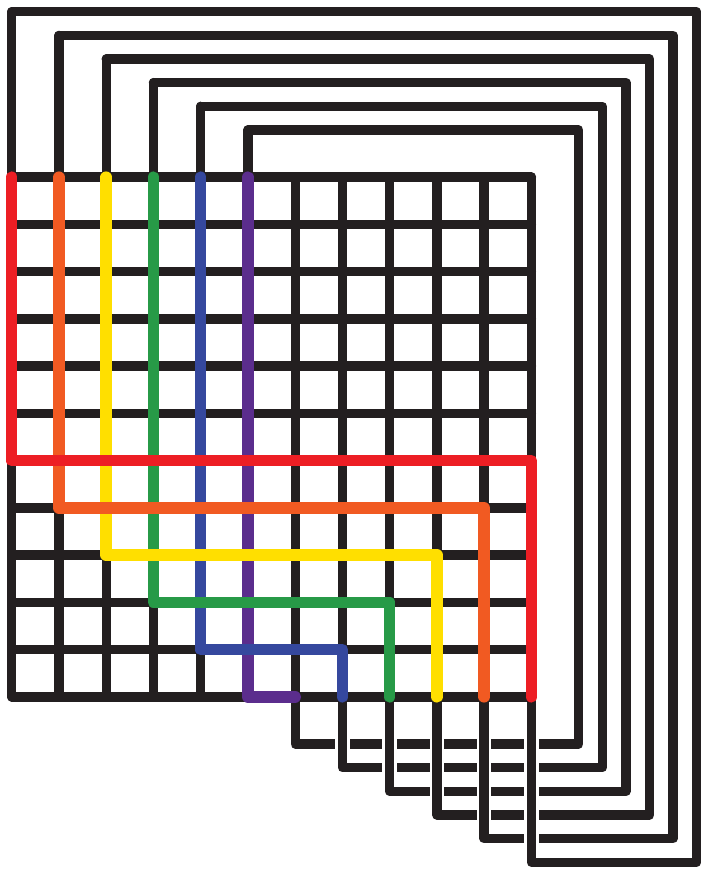}};
            \end{pgfonlayer}{background}
			
            \begin{pgfonlayer}{foreground}

            \end{pgfonlayer}{foreground}
        \end{tikzpicture}
    \caption{The graph $G$ and the half-integral packing of paths (indicated as a rainbow between the top left and the bottom right of the grid) used in the proof of~\zcref{thm:porousgrid}.}
    \label{fig:escherwall}
    \end{figure}
    
    To complete the proof, we now show that $G$ does not contain two vertex-disjoint $L$-cycles.  Let $C_1$ and $C_2$ be vertex-disjoint $L$-cycles in $G$.  Since $\sum_{a' \in A'} a' \notin L$ for all nonempty $A' \subseteq A$, $C_1$ and $C_2$ must each use an edge not in $W$.  For each $i \in [2]$, let $k_i$ be the maximum index such that $e_{k_i} \in E(C_i)$.   Suppose $e_j \in E(C_1)$ for some $j < k_1$. Note that 
    \[
    p_{k_1} = (p_{k_1}-\alpha)+\alpha < w(e_{k_1})+w(e_j) < w(C_1)  < \sum_{i=1}^{k_1} p_{i} + \alpha .
    \]
    Since $w(C_1) \in L$, this contradicts the fact that $p_{k_1}$ is $(\sum_{i=1}^{k_1-1} p_{i} + \alpha)$-lonely.  Thus, for each $i \in [2]$, $k_i$ is the unique index such that $e_{k_i} \in E(C_i)$.   Hence, $V(C_1) \cap V(C_2) \neq \emptyset$, since the paths $C_1 \setminus e_{k_1}$ and $C_2 \setminus e_{k_2}$ must intersect (by obvious topological considerations).
\end{proof}

By combining our two main theorems we also obtain the following instability result.  It says that for every $L \subseteq \mathbb{N}$ such that $L$-cycles have the Erd\H{o}s-P\'osa property, there is a set $L'$ which is `very close' to $L$ such that $L'$-cycles do not have the Erd\H{o}s-P\'osa property.  



\begin{corollary}
Let $L \subseteq \mathbb{N}$ be such that the set of $L$-cycles have the Erd\H{o}s-P\'osa property.  Then there exists $L' \subseteq \mathbb{N}$ such that $L'$ does not have the Erd\H{o}s-P\'osa property, and $L \Delta L'$ has lower density zero.  
\end{corollary}

\begin{proof}
If $L$ is finite, then let $L'$ be any infinite set of lower density zero containing $L$.  By~\zcref{lowerdensity}, the set of $L'$-cycles do not have the $\frac{1}{t}$-integral Erd\H{o}s-P\'osa property for all $t \in \mathbb{N}$, and clearly $\underline{d}(L \Delta L')=0$.  So, we may assume that $L$ is infinite.

For each $t \in \mathbb{N}$, let $X_t:=[3^t, 3^t+t] \cap L$.  Define $X:=\bigcup_{t \in \mathbb{N}} X_t$ and $L':=L \setminus X$.  By construction, $L'$ is porous.  Moreover, $L'$ is infinite since $L$ is not porous by~\zcref{porous}.  Clearly, $\underline{d}(L \Delta L')=0$ and the set of $L'$-cycles do not have the Erd\H{o}s-P\'osa property by~\zcref{porous}.  
\end{proof}

\section*{Acknowledgements} These results were obtained slightly before the \emph{Focused Workshop on Erdős-Pósa problems} in Będlewo, Poland, March 15-20, 2026. Nonetheless, we are very grateful to the organizers of the workshop for providing an excellent research environment and the impetus to write this paper.  

\addcontentsline{toc}{section}{References}
\bibliographystyle{alphaurl}
\bibliography{literature}

\end{document}